\documentclass[reqno,a4paper, 11pt]{amsart}

\usepackage[a4paper=true,pdfpagelabels]{hyperref}
\usepackage{graphicx}
\usepackage{mathabx}
\usepackage[ansinew]{inputenc}
\usepackage{amsfonts,epsfig}
\usepackage{latexsym}
\usepackage{amsmath}
\usepackage{amssymb}
\usepackage{color}

\newtheorem{theorem}{Theorem}
\newtheorem{lemma}[theorem]{Lemma}

\newtheorem{proposition}[theorem]{Proposition}

\newtheorem{lettertheorem}{Theorem}
\newtheorem{letterlemma}[lettertheorem]{Lemma}

\theoremstyle{definition}

\theoremstyle{remark}

\numberwithin{equation}{section}

\newcommand{\B}{\mathcal{B}}

\newcommand{\D}{\mathbb{D}}
\newcommand{\DD}{\widehat{\mathcal{D}}}
\newcommand{\Dd}{\widecheck{\mathcal{D}}}
\newcommand{\DDD}{\mathcal{D}}
\newcommand{\N}{\mathbb{N}}
\newcommand{\RR}{\mathbb{R}}

\newcommand{\C}{\mathbb{C}}

\newcommand{\e}{\varepsilon}
\newcommand{\ep}{\varepsilon}

\renewcommand{\phi}{\varphi}

\newcommand{\M}{\mathcal{M}}

\def\BMO{\mathord{\rm BMO}}
\def\MO{\mathord{\rm MO}}
\def\BO{\mathord{\rm BO}}
\def\BA{\mathord{\rm BA}}

\def\a{\alpha}       \def\b{\beta}        \def\g{\gamma}
           \def\e{\varepsilon}
     \def\om{\omega}      
       \def\t{\theta}       
                  \def\z{\zeta}

\def\Inv{{\mathcal Inv}}

\renewcommand{\H}{\mathcal{H}}

\newenvironment{Prf}{\noindent{\emph{Proof of}}}
{\hfill$\Box$ }

\begin{document}
\title[Bergman projection and BMO in hyperbolic metric]{Bergman projection and BMO in hyperbolic metric -- improvement of classical result}

\keywords{Bergman projection, Bergman metric, Bergman space, Bloch space, doubling weight, Hankel operator, mean oscillation.}

\thanks{The authors are indebted to Kehe Zhu for explaining how the special case $p=2$ works in the setting of the standard weighted Bergman spaces. The authors would also like to thank Antti Per\"al\"a for enlightening conversations on the topic.}

\thanks{The research was supported in part by Ministerio de Econom\'{\i}a y Competitividad, Spain, projects PGC2018-096166-B-100; La Junta de Andaluc{\'i}a, projects FQM210 S and UMA18-FEDERJA-002, and Vilho, Yrj\"o ja Kalle V\"ais\"al\"a foundation of Finnish Academy of Science and Letters.}

\author{Jos\'e \'Angel Pel\'aez}
\address{Departamento de An\'alisis Matem\'atico, Universidad de M\'alaga, Campus de
Teatinos, 29071 M\'alaga, Spain} 
\email{japelaez@uma.es}

\author[Jouni R\"atty\"a]{Jouni R\"atty\"a}
\address{University of Eastern Finland\\
Department of Physics and Mathematics\\
P.O.Box 111\\FI-80101 Joensuu\\
Finland}
\email{jouni.rattya@uef.fi}

\subjclass[2010]{Primary 30H20, 30H35, 47B34}

\date{\today}

\maketitle

\begin{abstract}
The Bergman projection $P_\alpha$, induced by a standard radial weight, is bounded and onto from $L^\infty$ to the Bloch space $\mathcal{B}$. However, $P_\alpha: L^\infty\to \mathcal{B}$ is not a projection. This fact can be emended via the boundedness of the operator $P_\alpha:\BMO_2(\Delta)\to\mathcal{B}$, where $\BMO_2(\Delta)$ is the space of functions of bounded mean oscillation in the Bergman metric.

We consider the Bergman projection $P_\omega$ and the space $\BMO_{\omega,p}(\Delta)$ of functions of bounded mean oscillation induced by $1<p<\infty$ and a radial weight $\omega\in\mathcal{M}$. Here $\mathcal{M}$ is a wide class of radial weights defined by means of moments of the weight, and it contains the standard and the exponential-type weights. We describe the weights such that $P_\omega:\BMO_{\omega,p}(\Delta)\to\mathcal{B}$ is bounded. They coincide with the weights for which $P_\omega: L^\infty \to \mathcal{B}$ is bounded and onto. This result seems to be new even for the standard radial weights when $p\ne2$.
\end{abstract}

\section{Introduction and main results}

It is well-known that the Bergman projection $P_\alpha$, induced by the standard weight $(\a+1)(1-|z|^2)^\alpha$, is bounded and onto from $L^\infty$ to the Bloch space $\B$~\cite[Section~5.1]{Zhu}. This is a very useful result with a large variety of applications in the operator theory on spaces of analytic functions on $\D$. However, the operator $P_\alpha:L^\infty\to\B$ is in fact not a projection because of the strict inclusion $H^\infty\subsetneq\B$. This downside 
can be emended by replacing $L^\infty$ by the space $\BMO_2(\Delta)$ of functions of bounded mean oscillation in the Bergman metric~\cite[Section~8.1]{Zhu}. It is known that the analytic functions in $\BMO_2(\Delta)$ constitute the Bloch space~$\B$~\cite[Theorem~8.7]{Zhu}, and it is a folklore result that $P_\alpha:\BMO_2(\Delta)\to\B$ is bounded.
Professor Kehe Zhu kindly offered us the following proof: 

If $f\in\BMO_2(\Delta)$, then the big Hankel operators $H^{\alpha}_f(g)=(I-P_\alpha)(fg)$ and 
$H^{\alpha}_{\overline{f}}(g)=(I-P_\alpha)(\overline{f}g)$ are both bounded on the Bergman space $A^2_\alpha$ by~\cite[Section~8.1]{Zhu}, and therefore so are the little Hankels $h^\alpha_f(g)=\overline{P_\alpha}(fg)$ and $h^\alpha_{\overline{f}}(g)=\overline{P_\alpha}(\overline{f}g)$. Now that $h^\alpha_{\overline f}=h^\alpha_{\overline{P_\alpha(f)}}$, and the little Hankel operator $h^{\alpha}_{\overline{\varphi}}$, induced by an analytic symbol $\varphi$, is bounded on $A^2_\alpha$ if and only if $\varphi\in\B$ by \cite[Section~8.7]{Zhu}, it follows that $P_\alpha(f)\in\B$, whenever $f\in\BMO_2(\Delta)$. Since this argument preserves the information on the norms, it follows that $P_\alpha:\BMO_2(\Delta)\to\B$ is bounded.
 
In this paper we are interested in understanding the nature of a space $X$ of complex-valued functions such that $X\cap\H(\D)=\B$, and radial weights $\omega$ for which the Bergman projection $P_\om:X\to\B$ is bounded. Here, as usual, $\H(\D)$ stands for the space of analytic functions in the unit disc $\D=\{z\in\C:|z|<1\}$. We proceed towards the statements via necessary notation. 

For a non-negative function $\om\in L^1([0,1))$, its extension to $\D$, defined by $\om(z)=\om(|z|)$ for all $z\in\D$, is called a radial weight. For $0<p<\infty$ and such an $\omega$, the Lebesgue space $L^p_\om$ consists of complex-valued measurable functions $f$ on $\D$ such that
    $$
    \|f\|_{L^p_\omega}^p=\int_\D|f(z)|^p\omega(z)\,dA(z)<\infty,
    $$
where $dA(z)=\frac{dx\,dy}{\pi}$ is the normalized Lebesgue area measure on $\D$. The corresponding weighted Bergman space is $A^p_\om=L^p_\omega\cap\H(\D)$. Throughout this paper we assume $\widehat{\om}(z)=\int_{|z|}^1\om(s)\,ds>0$ for all $z\in\D$, for otherwise $A^p_\om=\H(\D)$. For a radial weight $\om$, the orthogonal Bergman projection $P_\om$ from $L^2_\om$ to $A^2_\om$ is
		\begin{equation*}
    P_\om(f)(z)=\int_{\D}f(\z) \overline{B^\om_{z}(\z)}\,\om(\z)dA(\z),
    \end{equation*}
where $B^\om_{z}$ are the reproducing kernels of the Hilbert space $A^2_\om$.
It has been recently shown in \cite[Theorems~1-2-3]{PelRat2020} that the Bergman projection $P_\om$, induced by a radial weight $\omega$, is bounded from $L^\infty$ to the Bloch space $\B$ if and only if $\om\in\DD$, 
while the Bloch space is continuosly embedded into $P_\om(L^\infty)$ if and only if $\omega\in\M$. Therefore, $P_\om: L^\infty\to\B$ is bounded and onto if and only $\om\in\DDD=\DD\cap\M$. Recall that a radial weight $\om$ belongs to the class~$\DD$ if there exists a constant $C=C(\om)>1$ such that $\widehat{\om}(r)\le C\widehat{\om}(\frac{1+r}{2})$ for all $0\le r<1$, while $\om\in\M$ if $\om_{x}\ge C\om_{Kx}$, for all $x\ge1$, for some $C=C(\om)>1$ and $K=K(\om)>1$. Here and from now on $\om_x=\int_0^1r^x\om(r)\,dr$ for all $x\ge0$.

Let $\b(z,\z)$ denote the hyperbolic distance between the points $z$ and $\z$ in $\D$, and let $\Delta(z,r)$ stand for the hyperbolic disc of center $z\in\D$ and radius $0<r<\infty$. Further, let $\omega$ be a radial weight and $0<r<\infty$ such that $\omega\left(\Delta(z,r)\right)>0$ for all $z\in\D$. Then, for $f\in L^p_{\om,{\rm loc}}$ with $1\le p<\infty$, write 
    $$
    \MO_{\om,p,r}(f)(z)
		=\left(\frac{1}{\om(\Delta(z,r))}
		\int_{\Delta(z,r)}|f(\z)-\widehat{f}_{r,\om}(z)|^p\om(\z)\,dA(\z)\right)^{\frac{1}{p}},
    $$
where 
	$$
	\widehat{f}_{r,\om}(z)=\frac{\int_{\Delta(z,r)}f(\z)\om(\z)\,dA(\z)}{\om(\Delta(z,r))},\quad z\in\D.
	$$ 
The space $\BMO(\Delta)_{\om,p,r}$ consists of $f\in L^p_{\om}$ such that
    $$
    \|f\|_{\BMO(\Delta)_{\om,p,r}}=\sup_{z\in\D}\MO_{\om,p,r}(f)(z)<\infty.
    $$
It is known by \cite[Theorem~11]{PPR2020} that for each $\om\in\DDD$ there exists $r_0=r_0(\om)>0$ such that
	\begin{equation}\label{eq:intro1}
	\BMO(\Delta)_{\om,p,r}=\BMO(\Delta)_{\om,p,r_0}, \quad r\ge r_0.
	\end{equation}
We call this space $\BMO(\Delta)_{\om,p}$ whenever \eqref{eq:intro1} holds, and assume that the norm is always calculated with respect to a fixed $r\ge r_0$. However, in contrast to the class $\DDD$, for each prefixed $r>0$, the quantity $\om(\Delta(z,r))$ may equal to zero for some $z$ arbitrarily close to the boundary if $\om\in\DD$, by Proposition~\ref{le:counterexample} below. Therefore the space $\BMO(\Delta)_{\om,p,r}$ is not necessarily well-defined if $\om\in\DD$, and consequently, we consider the class $\DDD$ in the main results of this paper.

It is clear that the space $\BMO(\Delta)_{\om,p}$ depends on $\omega\in\DDD$, but for $\omega\in\Inv$,
straightforward calculations show that for each $r_1,r_2\in(0,\infty)$, we have $\BMO(\Delta)_{\om,p,r_1}=\BMO(\Delta)_{\nu,p,r_2}$ where $\nu(z)\equiv1$. Therefore we call this space $\BMO(\Delta)_{p}$. Recall that $\omega$ is invariant, denoted by $\om\in\Inv$, if for some (equivalently for all) $r\in(0,\infty)$ there exists a constant $C=C(r)\ge1$ such that such that $C^{-1}\om(\zeta)\le\omega(z)\le C\omega(\zeta)$ for all $\zeta\in\Delta(z,r)$. That is, an invariant weight is essentially constant in each hyperbolically bounded region. The class $\mathcal{R}$ of regular weights, which is a large subclass of smooth weights in $\DDD$, satisfies $\mathcal{R}\subset\Inv\cap\DDD$ by \cite[Section~1.3]{PR}. The space $\BMO(\Delta)_{\om,p}$ certainly depends on $p$ as is seen by considering the function $f(z)=|z|^{-\frac2p}$ which satisfies $f\in\BMO(\Delta)_q\setminus\BMO(\Delta)_p$ for $q<p$.  
 
We recall one last thing before stating the main result of this paper. Namely, an analytic function $f$ belongs to $\B$ if and only if it is Lipschitz continuous in the hyperpolic metric \cite[Theorem~5.5]{Zhu}. Therefore $\B\subset\BMO(\Delta)_{\om,p,r}$ for each $1\le p<\infty$, $0<r<\infty$ and a radial weight $\omega$ such that $\omega\left(\Delta(z,r)\right)>0$ for all $z\in\D$.

\begin{theorem}\label{th:main}
Let $1<p<\infty$ and $\om\in\M$. Then the following statements are equivalent:
	\begin{itemize}
	\item[\rm(i)] There exists $r_0=r_0(\omega)\in(0,\infty)$ such that $\BMO(\Delta)_{\om,p,r}$ does not depend on $r$, provided $r\ge r_0$. Moreover, $P_\om:\BMO(\Delta)_{\om,p}\to\B$ is bounded;
	\item[\rm(ii)] $P_\om:L^\infty\to\B$ is bounded;
	\item[\rm(iii)] $\om\in\DD$.
	\end{itemize}
\end{theorem}

The class $\M$ is a wide class of radial weights containing the standard radial weights as well as exponential-type weights~\cite[Chapter~1]{PR}. It is also worth observing that weights in $\M$ may admit a substantial oscillating behavior. In fact, a careful inspection of the proof of \cite[Proposition~14]{PelRat2020} reveals the existence of a weight $\omega\in\M$ such that $\BMO(\Delta)_{\om,p,r}$ is not well-defined for any $r>0$ and $1<p<\infty$, and therefore we cannot get rid of the first statement in the case (i) in Theorem~\ref{th:main}. However, each weight $\om$ in the class $\Dd$ has the property that $\om(\Delta(z,r))>0$ for all $z\in\D$ and for all $r$ sufficiently large depending on $\om$. The class $\Dd$ consists of radial weights $\omega$ for which there exist constants $K=K(\om)>1$ and $C=C(\om)>1$ such that $\widehat{\om}(r)\ge C\widehat{\om}\left(1-\frac{1-r}{K}\right)$ for all $0\le r<1$. Recall that $\DDD=\DD\cap\Dd=\DD\cap\M$ but $\Dd\subsetneq\M$ by~\cite[Proof of Theorem~3 and Proposition~14]{PelRat2020}.

We point out that, as far as we know, the statement in Theorem~\ref{th:main} is new even for the standard weights when $p\ne2$.
  
As for the proof of Theorem~\ref{th:main}, the equivalence between (ii) and (iii) is already known by \cite[Theorem~3]{PelRat2020}, so our contribution here consists of showing that (iii) implies (i). Our approach to this implication does not involve the big Hankel operators, though it has some similarities with the argument concerning the case $p=2$ explained in the beginning of this introduction. In particular, on the way to the proof, we establish a description of the general symbols $f$ such that the little Hankel operator 
	$$
  h^\om_{f}(g)(z)=\overline{P_\om}(fg)(z)=\int_\D f(\z)g(\z)B^\om_z(\z)\om(\z)\,dA(\z),\quad z\in\D,
  $$
is bounded on $A^p_\omega$, provided $1<p<\infty$ and $\omega\in\DDD$. In order to give the precise statement, we introduce the transformation 
	\begin{equation}\label{V}
	V_{\om,\nu}(f)(z)
	=\nu(z)\int_\D f(\zeta)\overline{B^{\om\nu}_z(\zeta)}\om(\zeta)\,dA(\zeta)
	,\quad z\in\D,\quad f\in L^1_\om,
	\end{equation} 
for each $\om\in\DDD$ and each radial function $\nu:\D\to [0,\infty)$ such that $\nu\omega$ is a weight. Moreover, 
we write $f\in L^1_{\om_{\log}}$ if
	\begin{equation}\label{fubinicondition}
  \|f\|_{L^1_{\om_{\log}}}=\int_\D|f(z)|\left(\log\frac{e}{1-|z|}\right)\om(z)\,dA(z)<\infty.
	\end{equation}

\begin{theorem}\label{Theorem:hankel-characterization}
Let $1<p<\infty$, $\om\in\DDD$ and $f\in L^1_{\om_{\log}}$. Then $h^\om_{\overline{f}}: A^p_\om\to A^p_\omega$ is bounded if and only if there exists $n_0=n_0(\om)\in\N$ such that, for each $n\ge n_0$, the weight $\nu(z)=(1-|z|)^n$ satisfies $V_{\om,\nu}(f)\in L^\infty$. Moreover, 
	$$
	\|V_{\om,\nu}(f)\|_{L^\infty}\asymp\|h^\om_{\overline{f}}\|_{A^p_\om\to\overline{A^p_\om}}
	$$
for each fixed $n\ge n_0$.
\end{theorem}

To simplify the notation, we write
	$$
	V_{\omega,n}=\{f:V_{\om,\nu}(f)\in L^\infty,\,\,\nu(z)=(1-|z|)^n\},\quad n\in\N.
	$$
It is worth mentioning that a radial function $f$ belongs to $V_{\omega,n}$ for some (equivalently for all) $n\in\N$ if and only if $f\in L^1_\omega$.
However, if $f\in L^1_{\om_{\log}}$ then $h^\om_{\overline{f}}=h^\om_{\overline{P_\omega(f)}}$, so $P_\omega(f)\in A^p_\omega$ is a necessary condition for $h^\om_{\overline{f}}: A^p_\om\to A^p_\omega$ to be bounded. Consequently, the hypothesis $f\in L^1_{\om_{\log}}$ of Theorem~\ref{Theorem:hankel-characterization} is justified in this sense.

Proposition~\ref{pr:Blochdescription} below implies that $f\in\H(\D)$ belongs to $V_{\om,n}$ for each $n\in\N$ if and only if $f\in\B$, and moreover, $\|f\|_{\B}\asymp\|V_{\om,\nu}(f)\|_{L^\infty}$ for each $\nu(z)=(1-|z|)^n$. Therefore,
in the case of anti-analytic symbols, we obtain the following neat result which is certainly related to the fact that, in many instances, the conditions that characterize the bounded Hankel (little or big) and integration operators, defined by 
	$$
	T_g(f)(z)=\int_0^zg'(\zeta)f(\zeta)\,d\zeta,
	$$
on weighted Bergman spaces (with $1<p<\infty$) are the same~\cite{Aleman:Constantin,DRWW2022,KR,PelSum14,PR,Zhu}. 

\begin{theorem}\label{proposition:Bloch-necessary}
Let $1<p<\infty$, $\om\in\DDD$ and $f\in\H(\D)$. Then $h^\om_{\overline{f}}:A^p_\om\to\overline{A^p_\om}$ is bounded if and only if $f\in\B$. Moreover,
	$$
	\|f\|_{\B}\asymp\|h^\om_{\overline{f}}\|_{A^p_\om\to\overline{A^p_\om}}.
	$$
\end{theorem}

It follows from the proof of Theorem~\ref{th:main}, that $P_\omega:X\to\B$ is bounded when $\omega\in\DDD$ and $X\subset L^1_{\om_{\log}}$ is a normed vector space such that $\|h^\om_{\overline{f}}\|_{A^p_\om\to\overline{A^p_\om}}\lesssim \|f\|_{X}$ for some $p\in (1,\infty)$. What follows from this is that $P_\om:L^1_{\om_{\log}}\cap V_{\om,n_0}\to\B$ is bounded and $L^1_{\om_{\log}}\cap V_{\omega,n_0}\cap\H(\D)=\B$, where $n_0=n_0(\om)\in\N$ is that of Theorem~\ref{Theorem:hankel-characterization}. In Proposition~\ref{th:Hankelqbiggerp} below, we establish the embedding 
	\begin{equation}\label{eq:intro2}
	\bigcup_{1<p<\infty}\BMO(\Delta)_{\om,p} \subset L^1_{\om_{\log}}\bigcap V_{\om,n_0} 
	\end{equation}
to prove that $h^\om_{\overline{f}}:A^s_\om\to\overline{A^s_\om}$ is bounded if $f\in \BMO(\Delta)_{\om,p}$, $1<s,p<\infty$. This result together with Theorems~\ref{Theorem:hankel-characterization} and \ref{proposition:Bloch-necessary} are some of the principal ingredients in the proof of Theorem~\ref{th:main}. Another crucial auxiliary result that we obtain is a precise norm-estimate for the modified Bergman reproducing kernels 
	$$
	\zeta\mapsto(1-\overline{z}\zeta)^N B^\omega_z(\zeta),\quad N\in\N,\quad \omega\in\DD,
	$$
in $A^p_\nu$ for any radial weight $\nu$. This is used to prove Theorems~\ref{Theorem:hankel-characterization} and \ref{proposition:Bloch-necessary}, and hence plays a fundamental role also in the proof of the main result.
One of the primary obstacles to obtain these kind of norm-estimates and, by and large, in the study of the Bergman projection induced by a radial weight, is the lack explicit expressions for $B^\omega_z$. 

We also observe that the statement in Theorem~\ref{proposition:Bloch-necessary} fails for $0<p\le1$. In the case $p=1$ the characterizing condition is
	\begin{equation}\label{msufk}
  \sup_{z\in\D}\left|f'(z)\right|(1-|z|)^2\log\frac{e}{1-|z|}<\infty,
  \end{equation}
while, for $0<p<1$, we have
	$$
	\|h^\om_{\overline{f}}\|_{A^p_\om\to\overline{A^p_\om}}
	\asymp\sup_{z\in\D}\frac{|f'(z)|(1-|z|^2)}{\left(\widehat{\om}(z)(1-|z|^2)\right)^{\frac{1}{p}-1}}
	$$
by \cite[Theorems~4 and 6]{DRWW2022}.

The letter $C=C(\cdot)$ will denote an absolute constant whose value depends on the parameters indicated
in the parenthesis, and may change from one occurrence to another.
We will use the notation $a\lesssim b$ if there exists a constant
$C=C(\cdot)>0$ such that $a\le Cb$, and $a\gtrsim b$ is understood
in an analogous manner. In particular, if $a\lesssim b$ and
$a\gtrsim b$, then we write $a\asymp b$ and say that $a$ and $b$ are comparable.

\section{Norm estimates for modified Bergman kernels}

The aim of this section is to establish the estimate, crucial in our arguments to prove the main results presented in the introduction, for the norm of the functions $\zeta\mapsto(1-\overline{z}\zeta)^NB^\om_z(\zeta)$ in $A^p_\nu$. We begin with known characterizations of weights in $\DD$. The next result follows by \cite[Lemma~2.1]{PelSum14} and its proof. 

\begin{letterlemma}\label{Lemma:weights-in-D-hat}
Let $\om$ be a radial weight. Then the following statements are equivalent:
\begin{itemize}
\item[\rm(i)] $\om\in\DD$;
\item[\rm(ii)] There exist $C=C(\om)>0$ and $\b=\b(\om)>0$ such that
    \begin{equation*}
    \begin{split}
    \widehat{\om}(r)\le C\left(\frac{1-r}{1-t}\right)^{\b}\widehat{\om}(t),\quad 0\le r\le t<1;
    \end{split}
    \end{equation*}
\item[\rm(iii)] There exists $\lambda=\lambda(\om)\ge0$ such that
    $$
    \int_\D\frac{\om(z)}{|1-\overline{\z}z|^{\lambda+1}}\,dA(z)\asymp\frac{\widehat{\om}(\zeta)}{(1-|\z|)^\lambda},\quad \z\in\D;
    $$
\item[\rm(iv)] There exist $C=C(\om)>0$ and $\eta=\eta(\om)>0$ such that
    \begin{equation*}
    \begin{split}
    \om_x\le C\left(\frac{y}{x}\right)^{\eta}\om_y,\quad 0<x\le
    y<\infty.
    \end{split}
    \end{equation*}
\end{itemize}
\end{letterlemma}

The following result of technical nature can be established by arguments similar to those used to prove \cite[(22)]{PR2016/1}. We leave the details of the proof for an interested reader.

\begin{lemma}\label{le:HLextended}
Let $\omega\in\DD$, $0<p<\infty$ and $\alpha\in\mathbb{R}$. Then
	$$
	\sum_{n=0}^\infty \frac{(n+1)^{\alpha-2}}{\omega^p_{2n+1}}s^n
	\asymp\int_0^s \frac{dt}{\widehat{\omega}(t)^p(1-t)^\alpha}+1,\quad 0<s<1.	
	$$
\end{lemma}

In the forthcoming arguments we will face weights which are of the same nature as $\om_{[\beta]}$ but slightly different. Since we have to keep track with the details here carefully, we introduce suitable notation for these creatures. For this purpose, we denote
	$$ 
	\left( \omega_{(n_1,y_1),(n_2,y_2),\dots,(n_k,y_k)}\right)_x
	=\int_{0}^1 r^x\left(\prod_{j=1}^k(1-r^{y_j})^{n_j}\right)\om(r)\,dr,\quad 0\le x<\infty,
	$$
for $k,n_j,y_j\in\N$ and $j\in\{1,\dots,k\}$. It follows from the inequality $1-r^a\le a(1-r)$, valid for all $1\le a<\infty$ and $0\le r\le1$, and \cite[(1.3)]{PelRat2020} that there exists a constant $C=C\left(\omega, \sum_{j=1}^k n_j\right)>0$ such that
	\begin{equation}
	\begin{split}\label{eq:momentogeneral}
	\left(\omega_{(n_1,y_1),(n_2,y_2),\dots,(n_k,y_k)}\right)_x
	&\le\left(\prod_{j=1}^k y_j^{n_j}\right)\left( \omega_{\left[\sum_{j=1}^k n_j \right]}\right)_x\\
	&\le C\left(\prod_{j=1}^k y_j^{n_j}\right)\frac{\omega_x}{x^{\sum_{j=1}^k n_j}}, \quad 0\le x<\infty.
	\end{split}
	\end{equation}

The kernel estimate that we are after reads as follows.

\begin{lemma}\label{kernelmix-general}
Let $2\le p<\infty$, $N\in\N$ and $\om\in\DD$, and let $\nu$ be a radial weight. Then there exists a constant $C=C(\om,\nu,p,N)>0$ such that
    \begin{equation}\label{eq:kernelmix-general}
    \int_\D|(1-\overline{z}\z)^NB^\om_z(\z)|^p\nu(\z)\,dA(\z)
		\le C\left(\int_0^{|z|}\frac{\widehat{\nu}(t)}{\widehat{\om}(t)^p(1-t)^{p(1-N)}}\,dt+1\right),\quad z\in\D.
    \end{equation}
Moreover, if $\nu\in\DDD$, then
		\begin{equation}\label{eq:kernelmix-generalX}
    \int_\D|(1-\overline{z}\z)^NB^\om_z(\z)|^p\nu(\z)\,dA(\z)
		\asymp\left(\int_0^{|z|}\frac{\widehat{\nu}(t)}{\widehat{\om}(t)^p(1-t)^{p(1-N)}}\,dt+1\right),\quad z\in\D.
    \end{equation}
\end{lemma}
		
\begin{proof}
First we will deal with the case $N=1$. A direct calculation shows that
    \begin{equation*}
    \begin{split}
    2(1-\overline{z}\z)B^\om_z(\z)
    &=(1-\overline{z}\z)\sum_{n=0}^\infty\frac{(\overline{z}\z)^n}{\om_{2n+1}}
    =\sum_{n=0}^\infty\frac{(\overline{z}\z)^n}{\om_{2n+1}}
		-\sum_{n=0}^\infty\frac{(\overline{z}\z)^{n+1}}{\om_{2n+1}}\\
    &=\frac1{\om_1}+\sum_{n=1}^\infty\left(\frac{\om_{2n-1}-\om_{2n+1}}{\om_{2n+1}\om_{2n-1}}\right)(\overline{z}\z)^n,
    \end{split}
    \end{equation*}
where
    \begin{equation}\label{eq:difmoments}
    \om_{2n-1}-\om_{2n+1}=\int_0^1r^{2n-1}\om(r)(1-r^2)\,dr=(\om_{(1,2)})_{2n-1},\quad n\in\N.
    \end{equation}
Thus
    \begin{equation}\label{eq:n=1}
    \begin{split}
    2(1-\overline{z}\z)B^\om_z(\z)
		&=\frac1{\om_1}+\sum_{n=1}^\infty\left(\frac{(\om_{(1,2)})_{2n-1}}{\om_{2n+1}\om_{2n-1}}\right)(\overline{z}\z)^n\\
		&=A_1(\omega)+\sum_{n=1}^\infty\left(\frac{(\om_{(1,2)})_{2n-1}}{\om_{2n+1}\om_{2n-1}}\right)(\overline{z}\z)^n ,\quad z,\z\in\D.
    \end{split}
    \end{equation}
By the Hardy-Littlewood inequality \cite[Theorem~6.3]{D} and \eqref{eq:momentogeneral}, we deduce
    \begin{equation*}\begin{split}
    \int_0^{2\pi}|(1-\overline{z}re^{i\t})B^\om_z(re^{i\t})|^p\,d\t
    & \lesssim A^p_1(\omega)+\sum_{n=1}^\infty\left(\frac{(\om_{(1,2)})_{2n-1}}{\om_{2n+1}\om_{2n-1}}\right)^p
		n^{p-2}|\z|^{np}|z|^{np}\\
		&\lesssim A^p_1(\omega)+\sum_{n=1}^\infty\left(\frac{1}{\om_{2n+1}}\right)^pn^{-2}r^{np}|z|^{np},
		\quad z\in \D,\quad 0\le r<1,
    \end{split}
    \end{equation*}
where the sum on the right-hand side corresponds to the case $\alpha=0$ of Lemma~\ref{le:HLextended}. Therefore we obtain
    $$
    \int_0^{2\pi}|(1-\overline{z}re^{i\t})B^\om_z(re^{i\t})|^p\,d\t
    \lesssim A^p_1(\omega)+\int_0^{|z|r}\frac{dt}{\widehat{\om}(t)^p},\quad z\in\D,\quad0\le r<1,
    $$
and it follows that
    \begin{equation}\label{pilipalipuli}
    \begin{split}
    \int_\D|(1-\overline{z}\z)B^\om_z(\z)|^p\nu(\z)\,dA(\z)
    &\lesssim\int_0^1\left(\int_0^{|z|r}\frac{dt}{\widehat{\om}(t)^p}\right)\nu(r)\,dr+1\\
		&=\int_0^{|z|}\frac{1}{\widehat{\om}(t)^p}\left(\int_{t/|z|}^1\nu(r)\,dr\right)\,dt+1\\
    &\le\int_0^{|z|}\frac{\widehat{\nu}(t)}{\widehat{\om}(t)^p}\,dt+1,\quad z\in\D.
    \end{split}
    \end{equation}
Therefore the case $N=1$ is proved.

Next we could now proceed by induction, but before doing that we discuss the case $N=2$ for the sake of clarity. By using \eqref{eq:n=1}, we obtain
		\begin{equation*}
    \begin{split}
		2(1-\overline{z}\z)^2B^\om_z(\z)
		&=(1-\overline{z}\z)\left(A_1(\omega)+\sum_{n=1}^\infty\frac{(\om_{(1,2)})_{2n-1} (\overline{z}\z)^n}{\om_{2n+1}\om_{2n-1}}\right)\\
		&=A_2(\omega,\overline{z}\z)+\sum_{n=2}^\infty
		\left(\frac{(\om_{(1,2)})_{2n-1}}{\om_{2n+1}\om_{2n-1}}-\frac{(\om_{(1,2)})_{2n-3}}{\om_{2n-1}\om_{2n-3}}\right)(\overline{z}\z)^n,\quad z,\z\in\D,
    \end{split}
    \end{equation*}
where  
	$$
	A_2(\omega,\overline{z}\z)=\frac{1-\overline{z}\z}{\omega_1}+\frac{(\om_{(1,2)})_{1}}{\omega_3\omega_1}\overline{z}\z,\quad z,\z\in\D.
	$$
Obviously, $|A_2(\omega,\overline{z}\z)|\le A_2(\omega)<\infty$ for all $z,\z\in\D$, and therefore it is a harmless term. Next, a reasoning similar to that in \eqref{eq:difmoments} gives
	\begin{equation}
	\begin{split}\label{eq:difmoments2}
	\left(\frac{(\om_{(1,2)})_{2n-1}}{\om_{2n+1}\om_{2n-1}}-\frac{(\om_{(1,2)})_{2n-3}}{\om_{2n-1}\om_{2n-3}}\right)
	&=\frac{1}{\om_{2n-1}}\left(\frac{(\om_{(1,2)})_{2n-1}}{\om_{2n+1}}-\frac{(\om_{(1,2)})_{2n-3}}{\om_{2n-3}}\right)\\
	&=\frac{1}{\om_{2n-1}}\bigg(\frac{(\om_{(1,2)})_{2n-1}-(\om_{(1,2)})_{2n-3}}{\om_{2n+1}}\\
	&\quad+(\om_{(1,2)})_{2n-3} \left(\frac{1}{\om_{2n+1}}-\frac{1}{\om_{2n-3}}\right)\bigg)\\
	&=\frac{1}{\om_{2n-1}}\left(-\frac{(\om_{(2,2)})_{2n-3}}{\om_{2n+1}}+(\om_{(1,2)})_{2n-3}
	\frac{\om_{2n-3}-\om_{2n+1}}{\om_{2n+1}\om_{2n-3}}\right)\\
	&=\frac{1}{\om_{2n-1}}\left(-\frac{(\om_{(2,2)})_{2n-3}}{\om_{2n+1}}
	+\frac{(\om_{(1,2)})_{2n-3}(\om_{(1,4)})_{2n-3}}{\om_{2n+1}\om_{2n-3}}\right)
	\end{split}
	\end{equation}    
for all $n\ge2$. Therefore
	\begin{equation}\label{eq:n=2}
  \begin{split}
	2(1-\overline{z}\z)^2B^\om_z(\z)
	=A_2(\omega,\overline{z}\z)-I(\omega,\overline{z}\z)+II(\omega,\overline{z}\z),\quad z,\z\in\D,
    \end{split}
    \end{equation}
where 
	$$
	I(\omega,\overline{z}\z)
	=\sum_{n=2}^\infty \frac{(\om_{(2,2)})_{2n-3}}{\om_{2n+1}\om_{2n-1}}(\overline{z}\z)^n,\quad z,\z\in\D,
	$$
and 
	$$
	II(\omega,\overline{z}\z)
	=\sum_{n=2}^\infty \frac{(\om_{(1,2)})_{2n-3}(\om_{(1,4)})_{2n-3}}{\om_{2n+1}\om_{2n-1}\om_{2n-3}}(\overline{z}\z)^n,
   \quad z,\z\in\D.
	$$
Next, by \eqref{eq:momentogeneral} and Lemma~\ref{Lemma:weights-in-D-hat}, we deduce
	\begin{equation*}
  \begin{split} 
	\frac{(\om_{(2,2)})_{2n-3}}{\om_{2n+1}\om_{2n-1}}
	\le C(\om,2)\frac{\om_{2n-3}}{(2n-3)^2\om_{2n+1}\om_{2n-1}}
	\lesssim\frac{1}{n^2\om_{2n+1}},\quad n\ge 2,
  \end{split}
  \end{equation*}
and 
  \begin{equation*}
	\begin{split} 
	\frac{(\om_{(1,2)})_{2n-3}(\om_{(1,4)})_{2n-3}}{\om_{2n+1}\om_{2n-1}\om_{2n-3}}
	\le C(\om,2)\frac{\omega_{2n-3}}{(2n-3)^2\om_{2n+1}\om_{2n-1}}
	\lesssim\frac{1}{n^2\om_{2n+1}},\quad n\ge 2.
  \end{split}
  \end{equation*}
Observe that even if $I$ and $II$ obey the same upper estimate, there is no significant cancellation in the difference $I-II$, and therefore we may consider them separately.

The Hardy-Littlewood inequality \cite[Theorem~6.3]{D} now yields
    \begin{equation*}
		\begin{split}
    \int_0^{2\pi}|(1-\overline{z}re^{i\t})^2B^\om_z(re^{i\t})|^p\,d\t
		\lesssim A^p_2(\omega)+\sum_{n=1}^\infty\left(\frac{1}{n\om_{2n+1}}\right)^pn^{-2}r^{np}|z|^{np},\quad z\in \D,\quad 0\le r<1,
    \end{split}
    \end{equation*}
where the sum on the right-hand side corresponds to the case $\alpha=-p$ of
Lemma~\ref{le:HLextended}. Therefore
    $$
    \int_0^{2\pi}|(1-\overline{z}re^{i\t})^2B^\om_z(re^{i\t})|^p\,d\t
    \lesssim A^p_2(\omega)+\int_0^{|z|r}\frac{dt}{\widehat{\om}(t)^p(1-t)^{-p}},\quad z\in \D,\quad 0\le r<1.
    $$
By arguing as in \eqref{pilipalipuli}, we obtain \eqref{eq:kernelmix-general} for $N=2$.    

Let us now proceed by induction. Let $N\in\N\setminus\{1\}$, and assume that there exist $M=M(N)\in\N$ and $L=L(N)\in\N$ such that
	\begin{equation}\label{eq:n=N}
  \begin{split}
	& 2(1-\overline{z}\z)^NB^\om_z(\z)
	=A_N(\omega,\overline{z}\z)+\sum_{j=1}^{M}\delta_jB^\omega_{N,\ell_j,z}(\zeta) \quad z,\z\in\D,
	\end{split}
  \end{equation} 
where $\delta_j=\pm1$, $\sup_{z,\z\in\D}|A_N(\omega,\overline{z}\z)|\le A_N(\omega)<\infty$, and 
    \begin{equation*}
    \begin{split}
    B^\omega_{N,\ell_j,z}(\zeta)
		=\sum_{k=N}^\infty
		\frac{\prod_{s\in\{m_1,\dots,m_{\ell_j}\}}\left(\omega_{(n^s_1,y^s_1),(n^s_2,y^s_2),\dots,(n^s_k,y^s_k)}\right)_{2k+1-2s}}
		{\prod_{m=0}^{\ell_j}\om_{2k+1-2m}}(\overline{z}\z)^k,\quad z,\z\in\D,
    \end{split}
    \end{equation*}   
for all $j\in\{1,\dots,M\}$, and
  \begin{equation}
	\begin{split}\label{LN}
	\max_{j=1,\dots,M}\ell_j\le N+1, \quad 
	\max_{s\in\{m_1,\dots,m_{\ell_j}\},\,j=1\dots,M}y^{s}_j\le L,\quad\text{and}\quad
  \sum_{j=1}^N\sum_{s\in\{m_1,\dots,m_{\ell_j}\}}n^{s}_j=N.
  \end{split}
	\end{equation}
 
We observe first that
		\begin{equation}\label{eq:N+1s1}
    \begin{split}
		(1-\overline{z}\z)B^\omega_{N,\ell,z}(\zeta)
		=A^{\ell}_{N+1}(\omega,\overline{z}\z)+\sum_{k=N+1}^\infty \widehat{B^\omega_{N+1,\ell}}(k)(\overline{z}\z)^k,\quad z,\z\in\D,
    \end{split}
    \end{equation}
where $\ell=\ell_j$ with $j\in\{1,\dots,M\}$, $\sup_{z,\z\in\D}|A^{\ell}_{N+1}(\omega,\overline{z}\z)|\le A^{\ell}_N(\omega)<\infty$ and    
		\begin{equation*}\label{eq:N+1s11}
		\begin{split}
    \widehat{B^\omega_{N+1,\ell}}(k)
		&=\frac{\prod_{s\in\{m_1,\dots,m_{\ell}\}}\left(\omega_{(n^s_1,y^s_1),(n^s_2,y^s_2),\dots,(n^s_k,y^s_k)}\right)_{2k+1-2s}}
		{\prod_{m=0}^{\ell}\om_{2k+1-2m}}\\
		&\quad-\frac{\prod_{s\in\{m_1,\dots,m_{\ell}\}}\left(\omega_{(n^s_1,y^s_1),(n^s_2,y^s_2),\dots,(n^s_k,y^s_k)}\right)_{2k-1-2s}}
		{\prod_{m=0}^{\ell}\om_{2k-1-2m}} 
    \end{split}
		\end{equation*}
Next, by generalizing the proof of \eqref{eq:difmoments2}, we obtain, for each $k\ge N+1$, the identity
		\begin{equation}\label{eq:N+1s2}
    \begin{split}
		\widehat{B^\omega_{N+1,\ell}}(k)\prod_{m=1}^{\ell}\om_{2k+1-2m}
		&=\frac{\prod_{s\in\{m_1,\dots,m_{\ell}\}}\left(\omega_{(n^s_1,y^s_1),(n^s_2,y^s_2),\dots,(n^s_k,y^s_k)}\right)_{2k+1-2s}}{ \om_{2k+1}}\\
		&\quad-\frac{\prod_{s\in\{m_1,\dots, m_{\ell}\}}\left(\omega_{(n^s_1,y^s_1),(n^s_2,y^s_2),\dots,(n^s_k,y^s_k)}\right)_{2k-1-2s}}
		{\om_{2k-1-2\ell}}\\
		&=\frac{1}{\om_{2k+1}}\bigg(\prod_{s\in\{m_1,\dots, m_{\ell}\}}\left(\omega_{(n^s_1,y^s_1),(n^s_2,y^s_2),\dots,(n^s_k,y^s_k)}\right)_{2k+1-2s}\\
		&\quad-\prod_{s\in\{m_1,\dots, m_{\ell}\}}\left(\omega_{(n^s_1,y^s_1),(n^s_2,y^s_2),\dots,(n^s_k,y^s_k)}\right)_{2k-1-2s}\bigg)\\
		&\quad+\frac{\left(\omega_{(1,2\ell)}\right)_{2k+1-2\ell}}{\om_{2k+1}\om_{2k+1-2\ell}} 
		\prod_{s\in\{m_1,\dots,m_{\ell}\}}\left(\omega_{(n^s_1,y^s_1),(n^s_2,y^s_2),\dots,(n^s_k,y^s_k)}\right)_{2k-1-2s}\\
		&=-\frac{1}{\om_{2k+1}}\sum_{j\in\{1,\dots \ell\}}D_j\\
		&\quad+\frac{\left(\omega_{(1,2\ell)}\right)_{2k+1-2\ell}}{\om_{2k+1}\om_{2k+1-2\ell}}
		\prod_{s\in\{m_1,\dots,m_{\ell}\}}\left(\omega_{(n^s_1,y^s_1),(n^s_2,y^s_2),\dots,(n^s_k,y^s_k)}\right)_{2k-1-2s},
    \end{split}
    \end{equation}
where
    \begin{equation*}
		\begin{split}
    D_j&=\left(\omega_{(1,2),(n^j_1,y^j_1),(n^j_2,y^j_2),\dots,(n^j_k,y^j_k)}\right)_{2k-1-2j}
		\prod_{s\in\{m_1,\dots m_{j-1}\}}\left( \omega_{(n^s_1,y^s_1),(n^s_2,y^s_2),\dots (n^s_k,y^s_k)}\right)_{2k-1-2s}\\
		&\quad\cdot\prod_{s\in\{m_{j+1}\dots m_{\ell}\}}\left( \omega_{(n^s_1,y^s_1),(n^s_2,y^s_2),\dots (n^s_k,y^s_k)}\right)_{2k+1-2s}.
		\end{split}
		\end{equation*}
In this last identity, a product without factors is considered equal to one. By combining \eqref{eq:N+1s1} and \eqref{eq:N+1s2} we obtain a formula of the type \eqref{eq:n=N} with $N+1$ in place of $N$. Therefore we have now proved \eqref{eq:n=N} for all $N\in\N$. Hence it follows that to get \eqref{eq:kernelmix-general}, it is enough to show that
		\begin{equation}\label{eq:kernelmix-generalell}
    \int_\D|B^\om_{N,\ell,z}(\z)|^p\nu(\z)\,dA(\z)
		\lesssim\int_0^{|z|}\frac{\widehat{\nu}(t)}{\widehat{\om}(t)^p(1-t)^{p(1-N)}}\,dt+1,\quad z\in\D,
    \end{equation}
for each $\ell=\ell_j$ with $j\in\{1,\dots,M\}$. But \eqref{eq:momentogeneral}, \eqref{LN} and Lemma~\ref{Lemma:weights-in-D-hat} yield
		\begin{equation}\label{eq:N+1s3}
    \begin{split}
		\frac{\prod_{s\in\{m_1,\dots,m_{\ell}\}}\left(\omega_{(n^s_1,y^s_1),(n^s_2,y^s_2),\dots,(n^s_k,y^s_k)}\right)_{2k+1-2s}}
		{\prod_{m=0}^{\ell}\om_{2k+1-2m}}
		\lesssim\frac{1}{\omega_{2k+1}k^{N}},
		\end{split}
		\end{equation}
Putting this together with \eqref{eq:N+1s1} and the Hardy-Littlewood inequality \cite[Theorem~6.3]{D}, we obtain
    \begin{equation*}\begin{split}
    \int_0^{2\pi}|B^\om_{N,\ell,z}(re^{i\t})|^p\,d\t
    \lesssim (A^{\ell}_{N}(\omega))^p
		+\sum_{n=1}^\infty\left(\frac{1}{n^{N-1}\om_{2n+1}}\right)^pn^{-2}r^{np}|z|^{np},\quad z\in\D,\quad0\le r<1,
    \end{split}
    \end{equation*}
where the right-hand side corresponds to  the case $\alpha=(1-N)p$ of Lemma~\ref{le:HLextended}. Therefore we deduce
    $$
    \int_0^{2\pi}|B^\om_{N,\ell,z}(re^{i\t})|^p\,d\t
    \lesssim (A^{\ell}_{N}(\omega))^p + \int_0^{|z|r}\frac{dt}{\widehat{\om}(t)^p(1-t)^{(1-N)p}},
		\quad z\in \D,\quad 0\le r<1,
    $$  
and it follows that
    \begin{equation*}
    \begin{split}
    \int_\D|B^\om_{N,\ell,z}(\z)|^p\nu(\z)\,dA(\z)
    &\lesssim\int_0^1\left(\int_0^{|z|r}\frac{dt}{\widehat{\om}(t)^p(1-t)^{(1-N)p}}\right)\nu(r)\,dr+1\\
		&=\int_0^{|z|}\frac{1}{\widehat{\om}(t)^p(1-t)^{(1-N)p}}\widehat{\nu}\left(\frac{t}{|z|}\right)\,dt+1\\
    &\le\int_0^{|z|}\frac{\widehat{\nu}(t)}{\widehat{\om}(t)^p(1-t)^{(1-N)p}}\,dt+1,\quad z\in\D.
    \end{split}
    \end{equation*}
Thus \eqref{eq:kernelmix-generalell} is satisfied and the proof of (i) is complete.

The assertion \eqref{eq:kernelmix-generalX} is an immediate consequence of \eqref{eq:kernelmix-general}, the trivial inequality $|1-\overline{z}\zeta|\ge1-|\zeta|$, the fact $\nu_{[Np]}\in\DDD$, which follows from Lemma~\ref{room}(i), and \cite[Theorem~1(ii)]{PR2016/1}. We underline here that this argument does not work for $\nu\in\DD$ in general, because $\nu_{[Np]}$ does not necessarily belong to $\DDD$ when $\nu\in\DD$ by \cite[Theorem~3]{PR2020}.      
\end{proof}

\section{Little Hankel operator}

We begin with an auxiliary result on weights. The following simple lemma is useful for our purposes. The first case says that $\DDD$ is closed under multiplication by the hat of another weight in $\DD$. The second case reveals that the same happens if the multiplication is done by a suitably small negative power of the hat.  For each radial weight $\om$ and $\b\in\RR$, write $\om_{[\b]}(z)=\om(z)(1-|z|)^\b$ for all $z\in\D$.

\begin{lemma}\label{room}
Let $\om\in\DDD$ and $\nu\in\DD$. Then the following statements hold: 
\begin{itemize}
\item[\rm(i)] $\omega\widehat{\nu}\in\DDD$ and $\widehat{\omega\widehat{\nu}}\asymp\widehat{\om}\widehat{\nu}$ on $[0,1)$;
\item[\rm(ii)] There exists $\gamma_0=\gamma_0(\om,\nu)>0$ such that, for each $\g\in(0,\g_0]$, we have $(\widehat{\nu})^{-\gamma}\om\in\DDD$, and 
	\begin{equation}\label{chochi}
	\int_r^1\frac{\om(s)}{\widehat{\nu}(s)^{\gamma}}\,ds\asymp\frac{\widehat{\om}(r)}{\widehat{\nu}(r)^{\gamma}},\quad 0\le r<1.
	\end{equation}
\end{itemize}
\end{lemma}

\begin{proof}
(i). It is clear that $\widehat{\omega\widehat{\nu}}\le\widehat{\om}\widehat{\nu}$ on $[0,1)$. Let $\beta=\beta(\nu)>0$ be that of Lemma~\ref{Lemma:weights-in-D-hat}(ii), and observe that $\om\in\Dd$ if
	\begin{equation}\label{Eq:d-check}
	\widehat{\om}(r)\le C'\int_r^{1-\frac{1-r}{K}}\om(t)\,dt,\quad0\le r<1,
	\end{equation}
for some constant $C'=C'(\om)>1$. Then the lemma together with \eqref{Eq:d-check} yields
	\begin{equation*}
	\begin{split}
	\widehat{\om\widehat{\nu}}(r)
	&\ge\frac{\widehat{\nu}(r)}{C(1-r)^\beta}\int_r^1\omega_{[\beta]}(s)\,ds
	\ge\frac{\widehat{\nu}(r)}{C(1-r)^\beta}\int_{r}^{1-\frac{1-r}{K}}\omega_{[\beta]}(s)\,ds\\
	&\ge\frac{\widehat{\nu}(r)}{C(1-r)^\beta}\frac{(1-r)^\beta}{K^\beta}\int_{r}^{1-\frac{1-r}{K}}\omega(s)\,ds
	\ge\frac{\widehat{\nu}(r)\widehat{\om}(r)}{CC'K^\beta}, \quad 0\le r<1,
	\end{split}
	\end{equation*}
and thus $\widehat{\omega\widehat{\nu}}\asymp\widehat{\om}\widehat{\nu}$ on $[0,1)$. Standard arguments involving this asymptotic equality, the definition of $\DD$ and \eqref{Eq:d-check} show that $\om\widehat{\nu}\in\DDD$.

(ii). By~\cite[(2.27)]{PelRat2020}, $\om\in\Dd$ if and only if there exist constants $C=C(\om)>0$ and $\a_0=\a_0(\om)>0$ such that
	\begin{equation}\label{Eq:Dd-characterization}
	\widehat{\om}(t)\le C\left(\frac{1-t}{1-r}\right)^\a\widehat{\om}(r),\quad 0\le r\le t<1,
	\end{equation}
for all $\alpha\in(0,\alpha_0]$. Let $\gamma=\gamma(\om,\nu)\in(0,\alpha_0/\beta)$, where $\beta=\beta(\nu)>0$ is that of Lemma~\ref{Lemma:weights-in-D-hat}(ii). Then
	$$
	\lim_{s\to1^-}\frac{\widehat{\om}(s)}{\widehat{\nu}(s)^{\gamma}}
	\lesssim
	\lim_{s\to1^-}\frac{(1-s)^{\alpha_0}}{\widehat{\nu}(s)^\gamma}=0.
	$$
Two integrations by parts together with \eqref{Eq:Dd-characterization} and Lemma~\ref{Lemma:weights-in-D-hat}(ii) yield 
	\begin{equation*}
	\begin{split}
	\frac{\widehat{\om}(r)}{\widehat{\nu}(r)^{\gamma}}
	&\le\int_r^1\frac{\om(s)}{\widehat{\nu}(s)^{\gamma}}\,ds
	=\frac{\widehat{\om}(r)}{\widehat{\nu}(r)^{\gamma}}
	+\gamma\int_r^1\frac{\widehat{\om}(s)}{\widehat{\nu}(s)^{\gamma+1}}\nu(s)\,ds\\
	&\lesssim\frac{\widehat{\om}(r)}{\widehat{\nu}(r)^{\gamma}}
	+\frac{\widehat{\om}(r)}{(1-r)^{\alpha_0}}
	\gamma\int_r^1(1-s)^{\alpha_0}\frac{\nu(s)}{\widehat{\nu}(s)^{\gamma+1}}\,ds\\
	&\lesssim\frac{\widehat{\om}(r)}{\widehat{\nu}(r)^{\gamma}}
	+\frac{\widehat{\om}(r)}{(1-r)^{\alpha_0}}
	\int_r^1\left(\frac{(1-s)^\beta}{\widehat{\nu}(s)}\right)^\gamma(1-s)^{\alpha_0-1-\gamma\beta}\,ds\\
	&\lesssim\frac{\widehat{\om}(r)}{\widehat{\nu}(r)^{\gamma}}
	+\frac{\widehat{\om}(r)}{(1-r)^{\alpha_0}}
	\left(\frac{(1-r)^\beta}{\widehat{\nu}(r)}\right)^\gamma\int_r^1(1-s)^{\alpha_0-1-\gamma\beta}\,ds
	\lesssim\frac{\widehat{\om}(r)}{\widehat{\nu}(r)^{\gamma}},\quad 0\le r<1.
	\end{split}
	\end{equation*}
Therefore \eqref{chochi} is satisfied, and standard arguments show that $(\widehat{\nu})^{-\gamma}\om\in\DDD$.
\end{proof}

Our second auxiliary result in this section shows that $h^\om_{\overline{f}}(g)=h^\om_{\overline{P_\om(f)}}(g)$ for all $g\in H^\infty$, provided $f\in L^1_{\om_{\log}}$ with $\om\in\DD$.

\begin{lemma}\label{lemma:hankel-P}
Let $1<p<\infty$, $\om\in\DD$ and $f\in L^1_{\om_{\log}}$. Then $h^\om_{\overline{f}}(g)=h^\om_{\overline{P_\om(f)}}(g)$ for all $g\in H^\infty$.
\end{lemma}

\begin{proof}
Tonelli's theorem and \cite[Theorem~1]{PR2016/1} yield	
	\begin{equation*}
	\begin{split}
	\left|h^\om_{\overline{P_\om(f)}}(g)(\xi)\right|
	&\le\|B_\xi^\om\|_{H^\infty}\|g\|_{H^\infty}\int_\D\left(\int_\D|f(\zeta)||B_z^\om(\zeta)|\om(\zeta)\,dA(\zeta)\right)\om(z)\,dA(z)\\
	&=\|B_\xi^\om\|_{H^\infty}\|g\|_{H^\infty}\int_\D|f(\zeta)|\left(\int_\D|B_\zeta^\om(z)|\om(z)\,dA(z)\right)\om(\zeta)\,dA(\zeta)\\
	&\lesssim\|B_\xi^\om\|_{H^\infty}\|g\|_{H^\infty}\|f\|_{L^1_{\om_{\log}}}<\infty,\quad \xi\in\D,
	\end{split}
	\end{equation*}
and therefore Fubini's theorem, and the fact that $g\in A^{p}_\om\subset A^1_\om$, yields
	\begin{equation}\label{suvlkjh}
	\begin{split}
	h^\om_{\overline{P_\om(f)}}(g)(\xi)
	&=\int_\D\left(\overline{\int_\D f(\zeta)\overline{B_z^\om(\zeta)}\om(\zeta)\,dA(\zeta)}\right)g(z)B_\xi^\om(z)\om(z)\,dA(z)\\
	&=\int_\D\overline{f(\zeta)}\left(\int_\D g(z)B_\xi^\om(z)\overline{B_\zeta^\om(z)}\om(z)\,dA(z)\right)\om(\zeta)\,dA(\zeta)\\
	&=\int_\D\overline{f(\zeta)}g(\zeta)B_\xi^\om(\zeta)\om(\zeta)\,dA(\zeta)
	=h^\om_{\overline{f}}(g)(\xi),\quad \xi\in\D.
	\end{split}
	\end{equation}
Thus $h^\om_{\overline{f}}(g)=h^\om_{\overline{P_\om(f)}}(g)$ for all $g\in H^\infty$ as claimed.
\end{proof}

The next lemma reveals that $\overline{P_\om V_{\om,\nu}(f)}=\overline{P_\om(f)}$ under appropriate hypotheses on $f$ and weights involved.

\begin{lemma}\label{lemma:V}
Let $\om$ be a radial weight and $\nu:\D\to[0,\infty)$ a radial function such that $\om\nu$ is a weight. Further, let $f:\D\to\C$ be measurable such that the function $z\mapsto f(z)\|B^{\om\nu}_z\|_{L^1_{\om\nu}}$ belongs to $L^1_\om$. Then $\overline{P_\om V_{\om,\nu}(f)}=\overline{P_\om(f)}$.
\end{lemma}

\begin{proof}	
Since the function $z\mapsto f(z)\|B^{\om\nu}_z\|_{L^1_{\om\nu}}$ belongs to $L^1_\om$ by the hypothesis, Fubini's theorem yields
	\begin{equation*}
	\begin{split}
	\overline{P_\om V_{\om,\nu}(f)(\zeta)}
	&=\int_\D\left(\nu(u)\int_\D\overline{f(v)}B^{\om\nu}_u(v)\om(v)\,dA(v)\right)B^\om_\zeta(u)\om(u)\,dA(u)\\
	&=\int_\D\overline{f(v)}\left(\int_\D B_\zeta^\om(u)\overline{B^{\om\nu}_v(u)}\om(u)\nu(u)\,dA(u)\right)\om(v)\,dA(v)\\
	&=\int_\D\overline{f(v)}B^\omega_\zeta(v)\om(v)\,dA(v)
    =\overline{P_\om(f)(\zeta)},\quad\zeta\in\D,
	\end{split}
	\end{equation*}
and thus the assertion is proved.
\end{proof}

With these preparations we can show that $V_{\om,\nu}(f)\in L^\infty$ is a sufficient condition for $h^\om_{\overline{f}}: A^p_\om\to\overline{A^p_\om}$ to be bounded. 

\begin{proposition}\label{Proposition:sufficient-V}
Let $1<p<\infty$, $\om\in\DDD$ and $f\in L^1_{\om_{\log}}$. Further, let $\nu:\D\to[0,\infty)$ be a radial function such that $\om\nu$ is a weight, and assume that the function $z\mapsto f(z)\|B^{\om\nu}_z\|_{L^1_{\om\nu}}$ belongs to $L^1_\om$.
If $V_{\om,\nu}(f)\in L^\infty$, then $h^\om_{\overline{f}}: A^p_\om\to 
\overline{A^p_\om}$ is bounded and
	$$
	\|h^\om_{\overline{f}}\|_{A^p_\om\to \overline{A^p_\om}}\lesssim\|V_{\om,\nu}(f)\|_{L^\infty}.
	$$
\end{proposition}

\begin{proof}
By Lemmas~\ref{lemma:hankel-P} and \ref{lemma:V}, and \cite[Theorem~9]{PelRat2020} we have
	\begin{equation*}
	\begin{split}
	\|h^\om_{\overline{f}}(g)\|_{L^{p}_\om}
	&=\|h^\om_{\overline{P_\om(f)}}(g)\|_{L^{p}_\om}
	=\|h^\om_{\overline{P_\om V_{\om,\nu}(f)}}(g)\|_{L^{p}_\om}
	=\|h^\om_{\overline{V_{\om,\nu}(f)}}(g)\|_{L^{p}_\om}\\
	&\le\|P_\om^+(|V_{\om,\nu}(f)g|)\|_{L^{p}_\om}
	\lesssim\|V_{\om,\nu}(f)g\|_{L^{p}_\om}
	\le\|V_{\om,\nu}(f)\|_{L^\infty}\|g\|_{A^{p}_\om},\quad g\in H^\infty,
	\end{split}
	\end{equation*}
and the assertion follows from the BLT-theorem~\cite[Theorem~I.7]{ReedSimon}, because $H^\infty$ is dense in $A^p_\om$ since $\om$ is radial.
\end{proof}

The necessity of the condition $V_{\om,\nu}(f)\in L^\infty$ for the boundedness of $h^\om_{\overline{f}}: A^p_\om\to\overline{A^p_\om}$ is established next.

\begin{proposition}\label{proposition:Bounded-necessary}
Let $1<p<\infty$, $\om\in\DDD$ and $f\in L^1_{\om_{\log}}$ such that $h^\om_{\overline{f}}:A^p_\om\to\overline{A^p_\om}$ is bounded. Then there exists $n_0=n_0(\omega)\in\N$ such that, for each $n\ge n_0$, the weight $\nu(z)=(1-|z|)^n$ satisfies $\|V_{\om,\nu}(f)\|_{L^\infty}\lesssim\|h^\om_{\overline{f}}\|_{A^p_\om\to\overline{A^p_\om}}$.
\end{proposition}

\begin{proof}
By Tonelli's theorem, \cite[Theorem~1]{PR2016/1} and the hypothesis $f\in L^1_{\om_{\log}}$, we deduce
	\begin{equation*}
	\begin{split}
	\int_\D\left(\int_\D|f(\zeta)||g(\zeta)||B_z^\om(\zeta)|\om(\zeta)\,dA(\zeta)\right)|h(z)|\om(z)\,dA(z)
	\lesssim\|g\|_{H^\infty}\|h\|_{H^\infty}\|f\|_{L^1_{\om_{\log}}}<\infty,
	\end{split}
	\end{equation*}
and hence Fubini's theorem yields
	\begin{equation}\label{eq:fubini}
	\begin{split}
	\left\langle \overline{h^\om_{\overline{f}}(g)},h\right\rangle_{A^2_\om}
	&=\int_\D\left(\overline{\int_\D\overline{f(\zeta)}g(\zeta)B_z^\om(\zeta)\om(\zeta)\,dA(\zeta)}\right)\overline{h(z)}\om(z)\,dA(z)\\
	&=\int_\D f(\zeta)\overline{g(\zeta)}\left(\int_\D \overline{B_z^\om(\zeta)h(z)}\om(z)\,dA(z)\right)\om(\zeta)\,dA(\zeta)\\
	&=\int_\D f(\zeta)\overline{g(\zeta)}\left(\overline{\int_\D h(z)\overline{B_\zeta^\om(z)}\om(z)\,dA(z)}\right)\om(\zeta)\,dA(\zeta)\\
	&=\int_\D f(\zeta)\overline{g(\zeta)h(\zeta)}\om(\zeta)\,dA(\zeta), \quad g\in H^\infty,\quad h\in H^\infty.
	\end{split}
	\end{equation}
Therefore H\"older's inequality and the boundedness of $h^\om_{\overline{f}}:A^p_\om\to\overline{A^p_\om}$ give
	\begin{equation}\label{jhf}
	\begin{split}
	\left|\int_\D f(\zeta)\overline{g(\zeta)h(\zeta)}\om(\zeta)\,dA(\zeta)\right|
	&=\left|\left\langle \overline{h^\om_{\overline{f}}(g)},h\right\rangle_{A^2_\om}\right|
	\le\left\|\overline{h^\om_{\overline{f}}(g)}\right\|_{A^{p}_\om}\|h\|_{A^{p'}_\om}\\
	&\le\|h^\om_{\overline{f}}\|_{A^p_\om\to\overline{A^p_\om}}\|g\|_{A^p_\om}\|h\|_{A^{p'}_\om}, \quad
	g\in H^\infty,\quad h\in H^\infty.
	\end{split}
	\end{equation}

If $p\ge 2$, then, for each $z\in\D$ and $n\in\N$, choose the functions $g,h\in H^\infty$ such that $g(\zeta)=B_z^{\om\nu}(\zeta)(1-\overline{z}\zeta)^n$ and $h(\zeta)=(1-\overline{z}\zeta)^{-n}$ for all $\zeta\in\D$. Then $gh=B^{\om\nu}_z$ on the left-hand side of \eqref{jhf}, and therefore
	\begin{equation*}
	\begin{split}
	\left|\int_\D f(\zeta)\overline{B^{\om\nu}_z}(\zeta)\om(\zeta)\,dA(\zeta)\right|
	&\le\|h^\om_{\overline{f}}\|_{A^p_\om\to\overline{A^p_\om}}
	\left(\int_\D|B^{\om\nu}_z(\zeta)(1-\overline{z}\z)^n|^p\om(\z)\,dA(\z)\right)^\frac1p\\
	&\quad\cdot\left(\int_\D\frac{\om(\zeta)}{|1-\overline{z}\zeta|^{np'}}\,dA(\zeta)\right)^\frac{1}{p'},\quad z\in\D.
	\end{split}
	\end{equation*}
Next, take $n_0=n_0(\omega)$ such that $n_0\ge\lambda+1$, where $\lambda=\lambda(\omega)\ge0$ is that of Lemma~\ref{Lemma:weights-in-D-hat}(iii). Then we have
	\begin{equation*}
	\int_\D\frac{\om(\zeta)}{|1-\overline{z}\zeta|^{np'}}\,dA(\zeta)
	\lesssim\frac{\widehat{\om}(z)}{(1-|z|)^{np'-1}},\quad z\in\D,
	\end{equation*}
for each fixed $n\ge n_0$. Moreover, since $\om\in\DDD$ by the hypothesis, $\om\nu=\om_{[n]}\in\DDD$ and $\widehat{\om_{[n]}}(z)\asymp\widehat{\om}(z)(1-|z|)^n$ for all $n\in\N$ by Lemma~\ref{room}. Therefore Lemma~\ref{kernelmix-general} 
 yields
	\begin{equation*}
	\begin{split}
	\int_\D|B^{\om\nu}_z(\zeta)(1-\overline{z}\z)^n|^p\om(\z)\,dA(\z)
	&\lesssim\int_0^{|z|}\frac{\widehat{\om}(t)}{\widehat{\om_{[n]}}(t)^p(1-t)^{p(1-n)}}\,dt+1\\
	&\asymp\int_0^{|z|}\frac{dt}{\widehat{\om}(t)^{p-1}(1-t)^{p}}+1\\
	&\asymp\frac1{\left(\widehat{\om}(z)(1-|z|)\right)^{p-1}},\quad z\in\D.
	\end{split}
	\end{equation*}
Thus
	\begin{equation}\label{eq:j2}
	\begin{split}
	\sup_{z\in\D}(1-|z|)^n\left|\int_\D f(\zeta)\overline{B^{\om\nu}_z(\zeta)}\om(\zeta)\,dA(\zeta)\right|
	&\lesssim\|h^\om_{\overline{f}}\|_{A^p_\om\to\overline{A^p_\om}},
	\end{split}
	\end{equation}
and the assertion follows.

If $p'>2$, then, for each $z\in\D$ and $n\in\N$, choose the functions $h,g\in H^\infty$ such that $h(\zeta)=B_z^{\om\nu}(\zeta)(1-\overline{z}\zeta)^n$ and $g(\zeta)=(1-\overline{z}\zeta)^{-n}$ for all $\zeta\in\D$. Then $gh=B^{\om\nu}_z$ in \eqref{jhf}, and by arguing as above we obtain \eqref{eq:j2}. This finishes the proof of the proposition. 
\end{proof}

By using Propositions~\ref{Proposition:sufficient-V} and~\ref{proposition:Bounded-necessary} it is now easy to prove Theorem~\ref{Theorem:hankel-characterization}, stated in the introduction.

\medskip

\begin{Prf}{\em{Theorem~\ref{Theorem:hankel-characterization}.}}
If $n\in\N$, then
	$$
	\|B^{\om_{[n]}}_z\|_{L^1_{\om_{[n]}}}\asymp\log\frac{e}{1-|z|},\quad z\in\D,
	$$
by Lemma~\ref{room}(i) and \cite[Theorem~1]{PR2016/1}. Therefore 
	\begin{equation*}
	\begin{split}
	\left\|f(\cdot)\left\|B^{\om_{[n]}}_{(\cdot)}\right\|_{L^1_{\om_{[n]}}}\right\|_{L^1_\om}
	&\asymp\|f\|_{L^1_{\om_{\log}}}<\infty
	\end{split}
	\end{equation*}
by the hypothesis $f\in L^1_{\om_{\log}}$. Consequently, the assertion of the theorem follows by Propositions~\ref{Proposition:sufficient-V} and~\ref{proposition:Bounded-necessary}.
\end{Prf}

\medskip

We now proceed towards the proof of Theorem~\ref{proposition:Bloch-necessary}. The following characterization of the Bloch space, which is interesting in its own right, is one of the key ingredients in the proof.

\begin{proposition}\label{pr:Blochdescription}
Let $f\in \H(\D)$ and $\omega,\nu\in\DDD$. Then $f\in\B$ if and only if $V_{\om,\widehat{\nu}}(f)\in L^\infty$. Moreover, 
	$$
	\|f\|_{\B}\asymp \|V_{\om,\widehat{\nu}}(f)\|_{L^\infty}.
	$$
\end{proposition}

\begin{proof}
Let first $f\in\B$. Then, by \cite[Theorem~3]{PelRat2020}, there exists $h\in L^\infty$ such that $P_\omega(h)=f$ and 
$\| h\|_{L^\infty}\asymp\|f\|_{\B}$. This together with Fubini's theorem, \cite[Theorem~1]{PR2016/1} and Lemma~\ref{room}(ii) yields
	$$
	\int_{\D} f(\zeta)\overline{B^{\om\widehat{\nu}}_z(\zeta)}\omega(\zeta)\,dA(\zeta)
	=\int_{\D}h(u)\overline{B^{\om\widehat{\nu}}_z(u)}\omega(u)\,dA(u),\quad z\in\D.
	$$
Therefore, by Lemma~\ref{room}(i), \cite[Theorem~1]{PR2016/1} and \eqref{Eq:Dd-characterization}, we deduce
	\begin{equation*}
	\begin{split}
	\left|\int_{\D} f(\zeta)\overline{B^{\om\widehat{\nu}}_z(\zeta)}\omega(\zeta)\,dA(\zeta) \right|
	&\le\|h\|_{L^\infty}\int_{\D}\left|B^{\om\widehat{\nu}}_z(u)\right|\omega(u)\,dA(u)\\
	&\lesssim \|f\|_{\B}\left(\int_{0}^{|z|} \frac{\widehat{\omega}(t)}{\widehat{\omega\widehat{\nu}}(t)(1-t)}\,dt+1 \right)\\
	&\asymp \|f\|_{\B}\left(\int_{0}^{|z|} \frac{dt}{\widehat{\nu}(t)(1-t)}+1\right)
	\lesssim\frac{\|f\|_{\B}}{\widehat{\nu}(z)},\quad z\in\D,
	\end{split}
	\end{equation*}
and hence $\|V_{\om,\widehat{\nu}}(f)\|_{L^\infty}\lesssim\|f\|_{\B}$ for each fixed $\om,\nu\in\DDD$. 

Conversely, assume that $\|V_{\om,\widehat{\nu}}(f)\|_{L^\infty}<\infty$. Let $k\in A^1_\omega$. Then the reproducing formula for functions in $A^1_{\omega\widehat{\nu}}$ and Fubini's theorem give
	\begin{equation*}
	\begin{split}
	\left|\left\langle k,f_r\right\rangle_{A^2_\om}\right|
	&=\left|\left\langle P_{\om\widehat{\nu}}k,f_r\right\rangle_{A^2_\om}\right|
	=\left| \int_{\D} k(u)\overline{V_{\om,\widehat{\nu}}(f_r)} \omega(u)\,du\right|
	\le\|k\|_{A^1_\omega}\|V_{\om,\widehat{\nu}}(f_r)\|_{L^\infty},\quad 0<r<1.
	\end{split}
	\end{equation*}
Since $\limsup_{r\to 1^-}\|V_{\om,\widehat{\nu}}(f_r)\|_{L^\infty}\le\|V_{\om,\widehat{\nu}}(f)\|_{L^\infty}$, it follows that
 	$$
	\left|\left\langle k,f\right\rangle_{A^2_\om}\right|\lesssim \|k\|_{A^1_\omega} \|V_{\om,\widehat{\nu}}(f)\|_{L^\infty},\quad k\in A^1_\omega.
	$$
But $\om\in\DDD$ by the hypothesis, and therefore $(A^1_\om)^\star$ is isomorphic to the Bloch space via the $A^2_\omega$-pairing with equivalence of norms~\cite[Theorem~3]{PelRat2020}. Hence $f\in\B$ with
	$$
	\|f\|_{\B}\lesssim\|V_{\om,\widehat{\nu}}(f)\|_{L^\infty}
	$$
for each fixed $\om,\nu\in\DDD$. This finishes the proof of the proposition.
\end{proof}
 
\begin{Prf}{\em{Theorem~\ref{proposition:Bloch-necessary}.}} Let first $f\in\B$. Then Proposition~\ref{pr:Blochdescription} yields 
	$$
	\|V_{\om,\nu}(f)\|_{L^\infty}\lesssim \|f\|_{\B}
	$$ 
for  each $\nu(z)=(1-|z|)^n$ with $n\in\N$. This together with the obvious embedding $\B\subset A^1_{\omega_{\log}}$ and Theorem~\ref{Theorem:hankel-characterization} shows that $h^\om_{\overline{f}}:A^p_\om\to\overline{A^p_\om}$ is bounded and $\|h^\om_{\overline{f}}\|_{A^p_\om\to\overline{A^p_\om}}\lesssim \|f\|_{\B}$.

Conversely, assume that $h^\om_{\overline{f}}:A^p_\om\to\overline{A^p_\om}$ is bounded. Then $\overline{h^\om_{\overline{f}}(1)}=f\in A^p_\omega\subset A^1_{\om_{\log}}$. Therefore, by Proposition~\ref{proposition:Bounded-necessary}, there exists $n_0=n_0(\omega)\in\N$ such that, for each $n\ge n_0$, the weight $\nu(z)=(1-|z|)^n$ satisfies $\|V_{\om,\nu}(f)\|_{L^\infty}\lesssim\|h^\om_{\overline{f}}\|_{A^p_\om\to\overline{A^p_\om}}$. This together with Proposition~\ref{pr:Blochdescription} implies $f\in\B$ with
	$$
	\|f\|_{\B}\lesssim \|V_{\om,\nu}(f)\|_{L^\infty}
	\lesssim\|h^\om_{\overline{f}}\|_{A^p_\om\to\overline{A^p_\om}}
	$$
for each fixed $n\ge n_0$.  
\end{Prf}

\section{Proof of Theorem~\ref{th:main}}

We begin the section by showing that the space $\BMO(\Delta)_{\om,p,r}$ is not necessarily well-defined for all $r\in(0,1)$ if $\om\in\DD\setminus\DDD$. This serves us as a justification for the initial hypotheses on $\om$ in our study.

\begin{proposition}\label{le:counterexample}
Let $\psi:[0,1)\to[0,\infty)$ be a differentiable function such that $2(1-r^2)\psi'(r)>1$ for all $r\in[0,1)$. Then there exist $\om=\om_\psi\in\DD\setminus\DDD$ and a sequence $\{r_n\}_{n=1}^\infty$ of points on $(0,1)$ depending on $\psi$ only such that $\lim_{n\to\infty}r_n=1$ and $\om_{\psi}(\Delta(r_n,\psi(r_n))=0$ for all $n\in\N$.
\end{proposition}

\begin{proof}
By the hypothesis on $\psi'$, $\psi:[0,1)\to[0,\infty)$ is an increasing function such that $\lim_{r\to1^-}\psi(r)=\infty$.
This also implies that the sequence $\{t_n\}_{n=1}^\infty$ defined inductively by the identities $t_1=0$ and $\beta(t_n,t_{n+1})=2\psi(t_{n+1})$ for all $n\in\N$ satisfies $\lim_{n\to\infty}t_n=1$. Then the annulus $\{z\in\D:t_n\le|z|\le t_{n+1}\}$ contains $\Delta(s_n,\psi(s_n))$, where $s_n$ is the hyperbolic midpoint of $(t_n,t_{n+1})$. Define $\om=\sum_{n=1}^\infty a_n\chi_{\{z:t_{2n}\le|z|\le t_{2n+1}\}}$, where $\{a_n\}_{n=1}^\infty$ are chosen such that $a_n(t_{2n+1}-t_{2n})=2^{-n}$ for all $n\in\N$. Then $\widehat{\om}(t_{2n})=\sum_{j=n}^\infty 2^{-j}=2^{1-n}$ for all $n\in\N$, and it follows that $\om\in\DD$ because $\beta\left(r,\frac{1+r}{2}\right)\asymp1$ for all $0\le r<1$, and $\beta(t_{2n},t_{2(n+1)})=2(\psi(t_{2n+1})+\psi(t_{2n+2}))\to\infty$, as $n\to\infty$. Further, by setting $r_n=s_{2n+1}$, we have $\om(\Delta_{h}(r_n,\psi(r_{n})))=0$ for all $n\in\N$. This also implies that $\om\not\in\DDD$. 
\end{proof}

Before proving the main result of this paper, some notation and previous results are needed. For continuous $f:\D\to\C$ and $0<r<\infty$, define
    $$
    \Omega_r f(z)=\sup\{|f(z)-f(\z)|:\beta(z,\z)<r\},\quad z\in\D,
    $$
and let $\BO(\Delta)$ denote the space of those $f$ such that
    $$
    \|f\|_{\BO(\Delta)}=\sup_{z\in\D}\Omega_r f(z)<\infty.
    $$
It is known that the definition of $\BO(\Delta)$ is independent of the choice of $r$ \cite[Lemma~8.1]{Zhu}.

Let $0<r<\infty$ and let $\omega$ be a radial weight such that $\omega\left(\Delta(z,r)\right)>0$ for all $z\in\D$. Then, for $0<p<\infty$, the space $\BA(\Delta)_{\om,p,r}$ consists of $f\in L^p_{\om,{\rm loc}}$ such that
    $$
    \|f\|_{\BA(\Delta)_{\om,p,r}}
    =\sup_{z\in\D}\left(\frac{1}{\om(\Delta(z,r))}\int_{\Delta(z,r)}|f(\z)|^p\om(\z)\,dA(\z)\right)^{\frac1p}<\infty.
    $$
If $\om\in\DDD$, then the space $\BA(\Delta)_{\om,p,r}$ depends on $p$ and $\om$ but, by \cite[Lemma~10]{PPR2020}, there exists an $r_0=r_0(\omega)\in(0,\infty)$ such that it is independent of $r$ as long as $r\ge r_0$, so we write $\BA(\Delta)_{\om,p}$ for short. 

The first result of this section shows that, for any $1<p,s<\infty$, $f\in\BMO(\Delta)_{\om,p}$ is a sufficient condition for $h_f^\om,h_{\overline{f}}^\om:A^s_\om\to\overline{A^s_\om}$ to be bounded.

\begin{theorem}\label{th:Hankelqbiggerp}
Let $1<p,s<\infty$ and $\om\in\DDD$. If $f\in\BMO(\Delta)_{\om,p}$, then $h_f^\om,h_{\overline{f}}^\om:A^s_\om\to\overline{A^s_\om}$ are bounded, and
	$$
	\|h_f^\om\|_{A^s_\om\to\overline{A^s_\om}}+\|h_{\overline{f}}^\om\|_{A^s_\om\to\overline{A^s_\om}}
	\lesssim\|f\|_{\BMO(\Delta)_{\om,p}}.
	$$
\end{theorem}

\begin{proof}
Let $f\in\BMO_{\om,p}(\Delta)$. In the proof we will use the fact that $f\in\BMO_{\om,p}(\Delta)$ if and only if it can be decomposed as $f=f_1+f_2$, where $f_1\in\BA(\Delta)_{\om,p}$ and $f_2\in\BO(\Delta)$ such that $\|f_1\|_{\BA(\Delta)_{\om,p}}+\|f_2\|_{\BO(\Delta)}\lesssim \|f\|_{\BMO(\Delta)_{\om,p}}$. This statement follows from \cite[Theorem~11(ii)]{PPR2020} and its proof. Let us denote $\nu(z)=(1-|z|)^n$ for all $z\in\D$. Since $f_1\in\BA(\Delta)_{\om,p}$, H\"older's inequality, \cite[Theorem~1]{PR2016/1} and Lemma~\ref{room}(i) imply
	\begin{equation*}
	\begin{split}
	|V_{\om,\nu}(f_1)(z)|
	&\le
	(1-|z|)^n\int_\D|f_1(\zeta)||B_z^{\om_{[n]}}(\zeta)|\om(\zeta)\,dA(\zeta)\\
	&\le(1-|z|)^n\left(\int_\D|f_1(\zeta)|^p|B_z^{\om_{[n]}}(\zeta)|\om(\zeta)\,dA(\zeta)\right)^\frac1p
	\left(\int_\D|B_z^{\om_{[n]}}(\zeta)|\om(\zeta)\,dA(\zeta)\right)^\frac1{p'}\\
	&\lesssim\|f_1\|_{\BA(\Delta)_{\om,p}}(1-|z|)^n\int_\D|B_z^{\om_{[n]}}(\zeta)|\om(\zeta)\,dA(\zeta)\\
	&\lesssim \|f_1\|_{\BA(\Delta)_{\om,p}}(1-|z|)^n\left(\int_{0}^{|z|} \frac{\widehat{\omega}(t)}{\widehat{\omega_{[n]}}(t)(1-t)}\,dt+1\right)\\
	&\asymp \|f_1\|_{\BA(\Delta)_{\om,p}}(1-|z|)^n\left(\int_{0}^{|z|} \frac{\widehat{\omega}(t)}{\widehat{\omega}(t)(1-t)^{n+1}}\,dt+1\right)
	\lesssim\|f_1\|_{\BA(\Delta)_{\om,p}},\quad z\in\D,
	\end{split}
	\end{equation*}
and hence
	\begin{equation}\label{eq:f1}
	\begin{split}
	\|V_{\om,\nu}(f_1)\|_{L^\infty}\lesssim \|f_1\|_{\BA(\Delta)_{\om,p}}.
	\end{split}
	\end{equation}
Moreover, by Lemma~\ref{room}(ii), there exists $\e=\e(\om)\in(0,\frac12)$ such that $\om_{[-\e]}\in\DDD$ and $\widehat{\om_{[-\e]}}\asymp\widehat{\om}_{[-\e]}$ on $\D$. Since $f_2\in\BO(\Delta)$, we have
	$$
	|f_2(z)-f_2(\z)|
	\lesssim(1+\beta(z,\z))\|f_2\|_{\BO(\Delta)} 
	\lesssim\frac{|1-\overline{z}\z|^{2\ep}}{(1-|z|)^{\ep}(1-|\z|)^{\ep}}\|f\|_{\BO(\Delta)},\quad z,\zeta\in\D.
	$$
Further, let $q\ge1/\e$. Then H\"older's inequality, Lemma~\ref{kernelmix-general} and \cite[Theorem~1]{PR2016/1} yield
	\begin{equation*}
	\begin{split}
	|V_{\om,\nu}(f_2)(z)|
	&\lesssim(1-|z|)^n\int_\D|f_2(\zeta)-f_2(z)||B_z^{\om_{[n]}}(\zeta)|\om(\zeta)\,dA(\zeta)
	+(1-|z|)^n|f_2(z)|\\
	&\lesssim\|f_2\|_{\BO(\Delta)}(1-|z|)^{n-\ep}\int_\D|(1-\overline{z}\zeta)B_z^{\om_{[n]}}(\zeta)|^{2\ep}|B_z^{\om_{[n]}}(\zeta)|^{1-2\e}|\om_{[-\e]}(\zeta)\,dA(\zeta)\\
	&\quad+\|f_2\|_{\BO(\Delta)}\\
	&\le\|f_2\|_{\BO(\Delta)}(1-|z|)^{n-\ep}\left(\int_\D|(1-\overline{z}\zeta)B_z^{\om_{[n]}}(\zeta)|^{2\e q}\om_{[-\e]}(\zeta)\,dA(\zeta)\right)^\frac1q\\
	&\quad\cdot\left(\int_\D|B_z^{\om_{[n]}}(\zeta)|^{(1-2\e)q'}\om_{[-\e]}(\zeta)\,dA(\zeta)\right)^\frac1{q'}+\|f_2\|_{\BO(\Delta)}\\
	&\lesssim\|f_2\|_{\BO(\Delta)}(1-|z|)^{n-\ep}\left(\int_0^{|z|}\frac{\widehat{\om_{[-\e]}}(t)}{(\widehat{\om_{[n]}}(t))^{2\e q}}\,dt+1\right)^\frac1q\\
	&\quad\cdot\left(\int_0^{|z|}\frac{\widehat{\om_{[-\e]}}(t)}{(\widehat{\om_{[n]}}(t))^{(1-2\e)q'}(1-t)^{(1-2\e)q'}}\,dt+1\right)^\frac1{q'}+\|f_2\|_{\BO(\Delta)}\\
	&\asymp\|f_2\|_{\BO(\Delta)}(1-|z|)^{n-\ep}\left(\int_0^{|z|}\frac{dt}{\widehat{\om}(t)^{2\e q-1}(1-t)^{n(2\e q)+\e}}\right)^\frac1q\\
	&\quad\cdot\left(\int_0^{|z|}\frac{\widehat{\om}(t)^{1-q'(1-2\e)}}{(1-t)^{(n+1)(1-2\e)q'+\e}}\,dt\right)^\frac1{q'}+\|f_2\|_{\BO(\Delta)}\\
	&\lesssim\|f_2\|_{\BO(\Delta)}(1-|z|)^{n-\ep}\left(\frac{1}{\widehat{\om}(z)^{2\e q-1}(1-|z|)^{n(2\e q)+\e-1}}\right)^\frac1q\\
	&\quad\cdot\left(\int_0^{|z|}\frac{\widehat{\om}(t)^{1-q'(1-2\e)}}{(1-t)^{(n+1)(1-2\e)q'+\e}}\,dt\right)^\frac1{q'}+\|f_2\|_{\BO(\Delta)}.
	\end{split}
	\end{equation*}
By \eqref{Eq:Dd-characterization}, there exists $n_0=n_0(\om)>0$ such that, for each $n\ge n_0$, we have 
	$$
	\int_0^{|z|}\frac{\widehat{\om}(t)^{1-q'(1-2\e)}}{(1-t)^{(n+1)(1-2\e)q'+\e}}\,dt
	\lesssim\frac{\widehat{\om}(z)^{1-q'(1-2\e)}}{(1-|z|)^{(n+1)(1-2\e)q'+\e-1}},\quad z\in\D.
	$$
By combining these estimates we deduce 
	\begin{equation}\label{eq:f2}
	\begin{split}
	\|V_{\om,\nu}(f_2)\|_{L^\infty}\lesssim \|f_2\|_{\BO(\Delta)}
	\end{split}
	\end{equation}
for each fixed $n\ge n_0$. Now that $\BMO(\Delta)_{\om,p}\subset L^1_{\om_{\log}}$ for each $1<p<\infty$, we may combine Theorem~\ref{Theorem:hankel-characterization} with \eqref{eq:f1} and \eqref{eq:f2}, to deduce that $h_{\overline{f}}^\om: A^s_\om\to\overline{A^s_\om}$ is bounded and
	\begin{equation*}
	\begin{split}
	\|h_{\overline{f}}^\om\|_{A^s_\om\to\overline{A^s_\om}}
	&\lesssim \|V_{\om,\nu}(f)\|_{L^\infty}
	\le\|V_{\om,\nu}(f_1)\|_{L^\infty}+ \|V_{\om,\nu}(f_2)\|_{L^\infty}\\
	&\lesssim\|f_1\|_{\BA(\Delta)_{\om,p}}+\|f_2\|_{\BO(\Delta)}
	\lesssim\|f\|_{\BMO(\Delta)_{\om,p}}.
	\end{split}
	\end{equation*}
In an analogous way we can show that $h_{f}^\om: A^s_\om\to\overline{A^s_\om}$ is bounded and $\|h_{f}^\om\|_{A^s_\om\to\overline{A^s_\om}}\lesssim\|f\|_{\BMO(\Delta)_{\om,p}}$.
\end{proof}

We can now finally pull everything together and prove the main result of this paper.

\medskip

\begin{Prf}{\em{Theorem~\ref{th:main}.}}
The statements (ii) and (iii) are equivalent by \cite[Theorem~1]{PelRat2020}. Assume now (iii), that is, $\om\in\DDD$. Then, by \cite[Theorem~11]{PPR2020}, there exists $r_0=r_0(\omega)\in(0,\infty)$ such that $\BMO(\Delta)_{\om,p,r}=\BMO(\Delta)_{\om,p}$ does not depend on $r$, provided $r\ge r_0$. Let $f\in \BMO(\Delta)_{\om,p}$. Then, by Theorem~\ref{th:Hankelqbiggerp}, $h_{\overline{f}}^\om: A^{s}_\om\to\overline{A^{s}_\om}$ is bounded and $\|h_{\overline{f}}^\om\|_{A^{s}_\om\to\overline{A^{s}_\om}}\lesssim\|f\|_{\BMO(\Delta)_{\om,p}}$ for each $1<s<\infty$. Therefore Theorem~\ref{proposition:Bloch-necessary} and  Lemma~\ref{lemma:hankel-P} together with the standard density argument yield $P_\omega(f)\in\B$ and 
  $$
	\|P_\omega(f)\|_{\B}
	\asymp\|h_{\overline{P_\om(f)}}^\om\|_{A^{s}_\om\to\overline{A^{s}_\om}}
  =\|h_{\overline{f}}^\om\|_{A^{s}_\om\to\overline{A^{s}_\om}}
	\lesssim\|f\|_{\BMO(\Delta)_{\om,p}}.
	$$
Therefore $P_\om:\BMO(\Delta)_{\om,p}\to\B$ is bounded. 

The fact that (i) implies (ii) is an immediate consequence of the continuous embedding $L^\infty\subset\BMO(\Delta)_{\omega,p}$.
\end{Prf}

\bibliographystyle{amsplain}

\end{document}